\documentclass[11pt,reqno]{amsart}
\usepackage{amsmath,amsxtra,latexsym,amsthm,amssymb,amscd,pb-diagram}
\usepackage{mathrsfs,mathabx,amsfonts}
\usepackage[colorlinks=true,linkcolor=magenta,citecolor=blue]{hyperref}
\usepackage[margin=1in]{geometry}

\usepackage{bm}
\usepackage{color}
\usepackage{makecell,multirow,diagbox}
\usepackage{mathtools}
\usepackage[displaymath, mathlines]{lineno}

\usepackage{graphicx}
\usepackage{epsfig}
\usepackage{array}
\usepackage{enumitem}

\newtheorem{definition}{Definition}[section]
\newtheorem{theorem}{Theorem}[section]
\newtheorem{proposition}{Proposition}[section]
\newtheorem{lemma}{Lemma}[section]

\newtheorem{remark}{Remark}[section]

\newcommand{\R}{\mathbb{R}}

\newcommand{\ity}{\infty}
\newcommand{\dps}{\displaystyle}
\newcommand{\f}{\displaystyle\frac}

\begin{document}
\title[Semi-linear $\sigma$-evolution equations with different damping types]{Effect of additional regularity for the initial data on semi-linear $\sigma$-evolution equations with different damping types}

\thanks{$^* $\textit{Corresponding author:} Dinh Van Duong (vanmath2002@gmail.com)}

\maketitle
\centerline{\scshape Dinh Van Duong$^{1, *}$, Tuan Anh Dao$^{1}$}
\medskip
{\footnotesize
	\centerline{$^1$ Faculty of Mathematics and Informatics, Hanoi University of Science and Technology}
	\centerline{No.1 Dai Co Viet road, Hanoi, Vietnam}}

\begin{abstract}
   In this paper, we would like to study the critical exponent for semi-linear $\sigma$-evolution equations with different damping types under the influence of additional regularity for the initial data. On the one hand, we establish the existence of global (in time) solutions for small initial data and the blow-up in finite time solutions in the supercritical case and the subcritical case, respectively. The very interesting phenomenon is that the critical case belonging to the global solution range or the blow-up solution range depends heavily on the assumption of additional regularity for the initial data. Furthermore, we are going to provide lifespan estimates for solutions when the blow-up phenomenon occurs.
\end{abstract}

\tableofcontents

\section{Introduction}
Let us consider the following Cauchy problem for semi-linear $\sigma$-evolution equations with different damping types:
\begin{equation} \label{Main.Eq.1}
\begin{cases}
u_{tt}+ (-\Delta)^\sigma u+ (-\Delta)^{\sigma_1} u_t+ (-\Delta)^{\sigma_2} u_t= |u|^p, &\quad x\in \R^n,\, t > 0, \\
u(0,x)= \varepsilon u_0(x),\quad u_t(0,x)= \varepsilon u_1(x), &\quad x\in \R^n, \\
\end{cases}
\end{equation}
where $\varepsilon$ is a positive real number, $\sigma\ge 1$ is assumed to be any fractional number and $0 \leq \sigma_1 < \sigma/2< \sigma_2 \leq \sigma$. The parameter $p>1$ represents the exponent of the power nonlinear term. The homogeneous Cauchy problem corresponding to the problem (\ref{Main.Eq.1}) we have in mind is
\begin{equation} \label{Main.Eq.2}
\begin{cases}
u_{tt}+ (-\Delta)^\sigma u+ (-\Delta)^{\sigma_1} u_t+ (-\Delta)^{\sigma_2} u_t= 0, &\quad x\in \R^n,\, t > 0, \\
u(0,x)= \varepsilon u_0(x),\quad u_t(0,x)= \varepsilon u_1(x), &\quad x\in \R^n.
\end{cases}
\end{equation}
The term $(-\Delta)^{\sigma}$ is referred to the $\sigma$-th power of the Laplace operator and the parameter $\sigma$ may be non-integer. In the non-integer case, it is determined by
$$ (-\Delta)^{\sigma} \psi (x) = \frak{F}^{-1}_{\xi \to x}\big(|\xi|^{2\sigma} \frak{F}_{x \to \xi}(\psi)(\xi)\big)(x), $$
where $\psi \in H^{2\sigma}$ and $\frak{F}$, $\frak{F}^{-1}$ denotes the Fourier transform, the inverse Fourier transform, respectively. The term $(-\Delta)^{\sigma} u$, with $\sigma \geq 1$, is known as a $\sigma$-evolution wave or a higher-order wave. The term $(-\Delta)^{\theta} u_t$ with $\theta \geq 0$ is called a damping term. More precisely, it is often referred to as the frictional (or external) damping, the parabolic like damping, the $\sigma$-evolution like damping and the visco-elastic (or strong) damping when $\theta = 0$, \,$\theta = \sigma_1 \in (0, \sigma/2)$,\, $\theta = \sigma_2 \in (\sigma/2, \sigma)$ and $\theta = \sigma$, respectively. \medskip

Let us review some historical perspectives on the study of the following Cauchy problem for damped $\sigma$-evolution equations:
\begin{align}
    \begin{cases}
        u_{tt} + (-\Delta)^{\sigma} u + (-\Delta)^{\theta} u_t = \mu |u|^p,  & x \in \mathbb{R}^n,\, t > 0,\\
        u(0,x) = u_0(x) ,\quad u_t(0,x) = u_1(x), &  x \in \mathbb{R}^n,
    \end{cases}\label{Eq.3}
\end{align}
where $\sigma \geq 1$, $\theta \in [0, \sigma]$ and $\mu \geq 0$. Matsumura \cite{Matsumura1976} was the first to establish the basic decay estimates for the problem (\ref{Eq.3}) with $(\sigma, \theta, \mu) = (1, 0, 0)$, also known as the damped wave equation. Subsequently, it has been established that the damped wave equation has a diffusive structure as $t \to \infty$. In recent years, many mathematicians have focused on studying a typical nonlinear problem of the semi-linear damped wave equation, which is the problem (\ref{Eq.3}) with $(\sigma, \theta, \mu) = (1,0, 1)$, for example, \cite{TodorovaYordanov2001, Ikeda2019, IkehataOhta2002, Hayashi2004, Hayashi2006, Hayashi2017}. The paper \cite{TodorovaYordanov2001} identified that the critical exponent of this problem for small initial data
 $$ (u_0, u_1) \in (H^1 \cap L^1) \times (L^2 \cap L^1) $$
is $p_{\rm Fuj} = 1+2/n$, also known as the Fujita critical exponent. Here, the critical exponent $p_{\rm crit}$ is understood as the threshold between global (in time) existence of small data weak solution and blow-up of solutions even for small data. Expanding the data space
 $$ (H^1 \cap L^m) \times (L^2 \cap L^m) $$
 for $m \in [1,2]$, the authors of the cited papers proved the global existence with small data when $p > 1+ 2m/n$, blow-up solution when $p \leq 1 + 2/n$ if $m = 1$ and $p < 1+ 2m/n$ if $m \in (1, 2]$. The investigation of the solution properties of the problem when the exponent of the nonlinear term $|u|^p$ takes the critical values is also of interest. The paper \cite{NakaoOno1993} studied the case $m = 2$ and they prove the global well-posedness with small data when $p \geq 1 + 4/n$. The paper \cite{Ikeda2019} extended this result by proving the existence of global mild solutions when $p \geq 1+ 2m/n$  along with some other conditions and the parameter $m$ satisfies 
  \begin{align*}
      m \in \left(\max\left\{1,\frac{\sqrt{n^2 + 16n}-n}{4}\right\}, 2\right].
  \end{align*}
For the generalization of the problem (\ref{Eq.3}), we have a series of papers \cite{DAbbiccoReissig2014, DAbbiccoEbert2021, DAbbiccoEbert2022, DuongKainaneReissig2015, DaoReissig2019_1, DaoReissig2019_2, DaoReissig2021, IkehataTakeda2019}. The papers \cite{DaoReissig2019_1, DaoReissig2019_2} obtained $L^1$ estimates for the solution to (\ref{Eq.3}) by using the theory of modified Bessel functions combined with Fa\`a di Bruno’s formula. However, this approach is only suitable for ``parabolic like model" corresponding to $\theta \in (0, \sigma/2)$ and ``$\sigma$-evolution-like model" corresponding to
$\theta \in (\sigma/2, \sigma)$. It is not effective when we consider the case of strongly damped $\sigma$-evolution equations, that is, $\theta = \sigma$. The authors of \cite{DAbbiccoEbert2021} have addressed this issue. They obtained optimal $L^m-L^q$ estimates, with $1\leq m \leq q \leq \infty$, for the solution of the linear problem by treating separately the two components of the solution, the oscillatory one and the diffusive one, in a particular $t$-dependent zone of the frequency space.  Additionally, they also identified that the critical exponent of the problem (\ref{Eq.3}) is
$$ p_{\rm crit}(m,\sigma,\theta): = 1+ \dfrac{2m\sigma}{n-m\min\{2\theta,\sigma\}} $$
by proving the global existence of solutions when $p > p_{\rm crit}(m,\sigma,\theta)$ under some additional conditions, and the existence of blow-up solutions when $p < p_{\rm crit}(m,\sigma,\theta)$ with small initial data in an appropriate space. Speaking about the value $p= p_{\rm crit}(m,\sigma,\theta)$, we can say that it has only been proved to belong to the blow-up range when $m = 1$, meanwhile, it still remains an open problem for any $m> 1$ so far. \medskip

Turning to our main problem (\ref{Main.Eq.1}), the critical exponent in some specific cases has also been found in several previous papers, such as \cite{IkehataTakeda2017} with $\sigma = 1$, $\sigma_1 = 0$ and $\sigma_2 = 1$,  \cite{DaoMichihisa2020} with $\sigma \geq 1$, $\sigma_1 = 0$ and $\sigma_2 =\sigma$, and \cite{ChenDAbbiccoGirardi2022} with $\sigma = 1$, $\sigma_1 \in [0, 1/2)$ and $\sigma_2 \in (1/2, 1]$. These papers are based on \cite{DaoDuongNguyen2024} and follows the method used in \cite{Ikeda2019}. In \cite{DaoDuongNguyen2024}, they achieved optimal decay rates of solutions to the linear problem (\ref{Main.Eq.2}) in the $L^2$ setting as well as indicated the global (in time) existence of small data solutions to (\ref{Main.Eq.1}) when
$$ p > 1+ \dfrac{2m\sigma}{n-2m\sigma_1} $$
with $m \in [1,2)$.  Moreover, the interaction between ``parabolic-like models" and ``$\sigma$-evolution-like models" in (\ref{Main.Eq.1}) and (\ref{Main.Eq.2}) is thoroughly analyzed in \cite{DaoDuongNguyen2024}. Regarding the purpose of this paper, on the one hand we aim to establish that the value
$$ p_{\rm crit}(m,\sigma,\sigma_1): = 1+ \dfrac{2m\sigma}{n-2m\sigma_1} $$
is exactly the critical exponent for the problem (\ref{Main.Eq.1}). On the other hand, we want to study the impact of the regularity condition $L^m$, with $m \geq 1$, on the solution properties to (\ref{Main.Eq.1}) when the exponent $p$ takes the corresponding critical values. Specifically, we concluded that when $m=1$, the critical exponent $p_{\rm crit}(1,\sigma,\sigma_1)$ belongs to the blow-up range, but when $m \in (1,2]$, the critical exponent $p_{\rm crit}(m,\sigma,\sigma_1)$ belongs to the global existence range in some low-dimensional spaces. This change occurs because the  nonlinearity $|u|^p$  decays faster than the linear part at the spatial infinity. More precisely, we can see that $|u|^p \in L^{m_1}$ (presented in Section \ref{section3}) with $m_1: = \max\{1, 2/p\} < m$, however, the linear part of the solution belongs to the space $L^m$. Therefore, we can control the nonlinearity even when the exponent $p$ takes the critical value. Additionally, the estimates for lifespan of blow-up solutions to (\ref{Main.Eq.1}) are also addressed in this work. \medskip
 
\textbf{Notations} \medskip

\begin{itemize}[leftmargin=*]
\item We write $f\lesssim g$ when there exists a constant $C>0$ such that $f\le Cg$, and $f \sim g$ when $g\lesssim f\lesssim g$.
\item As usual, $H^{a}$ and $\dot{H}^{a}$, with $a \ge 0$, denote Bessel and Riesz potential spaces based on $L^2$ spaces. Here $\big<D\big>^{a}$ and $|D|^{a}$ stand for the differential operators with symbols $\big<\xi\big>^{a}$ and $|\xi|^{a}$, respectively.
\item For any $\gamma \in \R$, we denote by $[\gamma]^+:= \max\{\gamma,0\}$, its positive part.
\item Finally, we introduce the space
$\mathcal{D}_{\sigma}^m:= \big(H^{\sigma}\cap L^m\big) \times \big(L^2\cap L^m\big)$ with the norm
$$\|(\phi_0,\phi_1)\|_{\mathcal{D}_{\sigma}^m}:=\|\phi_0\|_{H^{\sigma}}+ \|\phi_0\|_{L^m}+ \|\phi_1\|_{L^2}+ \|\phi_1\|_{L^m} \quad \text{ with } m\in [1,2]. $$
\end{itemize}
\medskip

\textbf{Main results} \medskip

Let us state the local and global (in time) existence of small data solutions, along with the blow-up results, which will be proved in this paper. \medskip

\begin{theorem}[\textbf{Local existence}]\label{Local-existence}
    Let $n \geq 1$ and $0 \leq \sigma_1 < \sigma/2 < \sigma_2 \leq \sigma$. We assume that the parameter $m$ satisfies
    \begin{align} \label{condition_Local_1.1.1}
      1 < m \begin{cases}
          \leq 2 &\text{ if } \sigma_1 = 0,\\
          < \displaystyle\frac{2n}{n +4\sigma_1} &\text{ if } \sigma_1 > 0,
      \end{cases}  
    \end{align}
    and the exponent $p$ satisfies the following conditions:
    \begin{align}\label{condition_Local_1.2}
        \displaystyle\frac{2}{m} < p
        \begin{cases}
        \vspace{0.2cm}
            < \infty &\text{ if } n \leq 2\sigma,\\
            \leq \displaystyle\frac{n}{n -2\sigma} &\text{ if } 2\sigma < n < \displaystyle\frac{4\sigma}{2-m}.
        \end{cases}
    \end{align}
    Moreover, if the initial data $(u_0, u_1) \in \mathcal{D}_{\sigma}^m$, then there exists $T > 0 $ and a constant $\varepsilon_0 > 0$ such that for any $\varepsilon \in (0, \varepsilon_0]$, the Cauchy problem \eqref{Main.Eq.1} admits a unique local (in time) solution belonging to the class
\begin{equation*}
    u \in \mathcal{C}\left([0, T), H^{\sigma}\right).
\end{equation*}
\end{theorem}

\begin{theorem}[\textbf{Global existence}]\label{Global-Existance}
    Let $n \geq 1$ and $0 \leq \sigma_1 < \sigma/2 < \sigma_2 \leq \sigma$. The parameter $m$ satisfies
    \begin{equation}\label{condition 1.3.1}
      \max\left\{1,\frac{\sqrt{(n+4\sigma_1)^2+16n(\sigma-\sigma_1)}-n-4\sigma_1}{4(\sigma-\sigma_1)}\right\} < m
      \begin{cases}
      \vspace{0.2cm}
      \leq 2 &\text{ if } \sigma_1 = 0,\\
      <\displaystyle \frac{2n}{n+4\sigma_1} &\text{ if } \sigma_1 > 0.
      \end{cases}
    \end{equation}
    Moreover, we assume the following conditions:
    \begin{align}\label{condition_global_1.2.1}
    1+ \displaystyle\frac{2m\sigma}{n-2m\sigma_1} \leq p
        \begin{cases}
        \vspace{0.2cm}
            < \infty &\text{ if } n \leq 2\sigma\\
             \leq  \displaystyle\frac{n}{n-2\sigma} &\text{ if }  2\sigma < n < \displaystyle\frac{4\sigma}{2-m}
        \end{cases}
    \end{align}
    and the initial data $(u_0, u_1) \in \mathcal{D}_{\sigma}^m$. Then, there exists a constant $\varepsilon_0 > 0$ such that for any $\varepsilon \in (0,\varepsilon_0]$, the Cauchy problem \eqref{Main.Eq.1} admits a unique global solution belonging to the class
\begin{equation*}
    u \in \mathcal{C}\left([0, \infty\right), H^{\sigma}).
\end{equation*}
Furthermore, the following estimates hold:
\begin{align*}
    \|u(t,\cdot)\|_{L^2} &\lesssim (1+t)^{-\frac{n}{2(\sigma-\sigma_1)}(\frac{1}{m}-\frac{1}{2})+\frac{\sigma_1}{\sigma-\sigma_1}} \|(u_0,u_1)\|_{\mathcal{D}_{\sigma}^m},\\
    \||D|^{\sigma}u(t,\cdot)\|_{L^2} &\lesssim (1+t)^{-\frac{n}{2(\sigma-\sigma_1)}(\frac{1}{m}-\frac{1}{2})-\frac{\sigma-2\sigma_1}{2(\sigma-\sigma_1)}} \|(u_0,u_1)\|_{\mathcal{D}_{\sigma}^m}.
\end{align*}
\end{theorem}

\begin{remark}
    \fontshape{n}
\selectfont
We will explain the appearance of these conditions appearing in Theorems \ref{Local-existence}-\ref{Global-Existance}. The second term on the left-hand side of the condition (\ref{condition 1.3.1}) follows from the following inequality:
\begin{align*}
    1 + \frac{2m\sigma}{n-2m\sigma_1} > \frac{2}{m},
\end{align*}
which is to guarantee the condition \eqref{condition_Local_1.2} for the local solution existence. Furthermore, we can simplify the following condition in some low-dimensional spaces, namely,
\begin{align*}
     \max\left\{1,\frac{\sqrt{(n+4\sigma_1)^2+16n(\sigma-\sigma_1)}-n-4\sigma_1}{4(\sigma-\sigma_1)}\right\} = 1
\end{align*}
if and only if $n \leq 2(\sigma +\sigma_1)$. On the other hand, the right-hand sides of the conditions (\ref{condition_Local_1.1.1}) and (\ref{condition 1.3.1}) arise from the application of Lemma \ref{Proposition3.1}. More be specific, it is derived from the condition $n > 2m_0\sigma_1$
  appearing in Theorem 2.1 of \cite{DaoDuongNguyen2024}. Therefore, in the case $\sigma_1 = 0$, this condition disappears, leading to the validity of  $m = 2$. Finally, the appearance of the right-hand side in the conditions (\ref{condition_Local_1.2}) and (\ref{condition_global_1.2.1}) comes from the use of Proposition \ref{lemma3.2}.
\end{remark}

\begin{remark}
 \fontshape{n}
\selectfont
    From the statements of Theorems \ref{Local-existence} and \ref{Global-Existance}, one sees that the damping term $(-\Delta)^{\sigma_2} u_t$, that is, ``$\sigma$-evolution like damping", does not have a remarkable impact on the existence of either local or global (in time) solutions to the problem (\ref{Main.Eq.1}). The main reason is the choice of the solution space $X(T)$ and the weights in the definition of its norm, which are discussed in the proof of Theorems \ref{Local-existence} and \ref{Global-Existance} in Section \ref{section3}.
\end{remark}
\medskip

\begin{theorem}[\textbf{Blow-up}]\label{Blow-up results}
    Let $n \geq 1$, $m\ge 1$ and $ 0 \leq \sigma_1 < \sigma/2 < \sigma_2 \leq \sigma$. We assume that the data $u_0 \equiv 0$ and $u_1 \in L^m$ satisfying the following relation:
    \begin{equation}\label{condition1.1.1}
        \displaystyle\int_{\mathbb{R}^n} u_1(x)dx> 0 \,\,\text{ if } m =1 \quad \text{ or }\quad u_1(x) \geq \langle x \rangle^{-\frac{n}{m}} (\log(e+|x|))^{-1} \,\,\text{ if } m > 1.
    \end{equation}
     Moreover, we also assume that the exponent $p$ and the spatial dimension $n$ fulfill
    \begin{align}
    1 < p
    \begin{cases}\label{condition1.1.2}
        \vspace{0.2cm}
        <\ity &\text{ if } m\ge 1 \text{ and } n \leq 2m\sigma_1, \\
    \vspace{0.2cm}
        \leq 1 + \displaystyle\frac{2\sigma}{n-2\sigma_1} &\text{ if } m = 1 \text{ and } n> 2\sigma_1, \\
        \vspace{0.2cm}
        < 1+\displaystyle\frac{2m\sigma}{n-2m\sigma_1} &\text{ if } m > 1 \text{ and } n> 2m\sigma_1. 
    \end{cases}
    \end{align}
Then, there is no global (in time) weak solution $u \in \mathcal{C}([0,\infty), L^2)$ to \eqref{Main.Eq.1}.
\end{theorem}

\begin{remark}
\fontshape{n}
\selectfont
Based on Theorems \ref{Global-Existance} and \ref{Blow-up results}, and Theorem 3.1 in \cite{DaoDuongNguyen2024}, which addresses global existence for the case $m=1$, we can observe the influence of the 
$L^m$ regularity of the initial data $(u_0, u_1)$ and the exponent $p$ on the solution properties of the problem (\ref{Main.Eq.1}) as follows:
\begin{itemize}[leftmargin=*]
    \item If $m = 1$, the problem (\ref{Main.Eq.1}) has a unique global (in time) solution with small data when
    $$p > 1+ \frac{2\sigma}{n-2\sigma_1} $$
    and admits a blow-up solution when
    $$p \leq 1+ \frac{2\sigma}{n-2\sigma_1}.$$
    This means that the critical exponent $p_{\rm crit}(1,\sigma,\sigma_1)= 1+ (2\sigma)/(n-2\sigma_1)$ belongs to the blow-up range.
    \item If $m$ satisfies the condition (\ref{condition 1.3.1}), the problem (\ref{Main.Eq.1}) has a unique global (in time) solution with small data when
    $$ p \geq 1+ \frac{2m\sigma}{n-2m\sigma_1} $$
    and admits a blow-up solution when
    $$ p < 1+ \frac{2m\sigma}{n-2m\sigma_1}. $$
    This means that the critical exponent $p_{\rm crit}(m,\sigma,\sigma_1)= 1+ (2m\sigma)/(n-2m\sigma_1)$ belongs to the global existence range.
\end{itemize}
\end{remark}

\begin{remark}
\fontshape{n}
\selectfont
    In this remark, let us consider the special case of $\sigma_1 = 0$ corresponding to frictional damping and $\sigma_2 =\sigma$ corresponding to visco-elastic damping, which has been studied in \cite{DaoMichihisa2020, IkehataTakeda2017}. In this context, the authors identified that the critical exponent for this problem is
\begin{align*}
    p_{\rm crit}(m,\sigma):= 1+ \frac{2m\sigma}{n}.
\end{align*}
More in detail, they established the existence of a global (in time) solution with small initial data when $p > p_{\rm crit}(m,\sigma)$ and a blow-up solution when $p < p_{\rm crit}(m,\sigma)$. However, the critical case $p = p_{\rm crit}(m,\sigma)$ was not addressed in these cited papers. For this reason, Theorem \ref{Global-Existance} fills this gap from the view of critical exponents by demonstrating the existence of a global solution when $p=p_c(m)$, under the following condition:
$$
   m \in (1,2] \text{ if } n \leq 2\sigma \,\text{ or } m \in \left(\frac{\sqrt{n^2 + 16n \sigma} -n}{4\sigma}, 2\right] \text{ if } n > 2\sigma.
$$
\end{remark}

\textbf{This paper is organized as follows:} In Section \ref{section3}, we provide the proofs of the local and global (in time) existence results for solutions to the problem (\ref{Main.Eq.1}). Subsequently, in Section \ref{Proof of blow-up results} we establish the blow-up result and derive lifespan estimates for solutions in the subcritical case as well.

\section{Proofs of local and global (in time) existence}\label{section3}
\subsection{Philosophy of our approach}
In this section, let us focus on proving Theorems \ref{Local-existence} and \ref{Global-Existance}. First, we can write the solution to $(\ref{Main.Eq.2})$ by the formula
\begin{equation*}
u(t,x) = \varepsilon K_0(t,x) \ast_{ x} u_0(x) + \varepsilon K_1(t,x) \ast_{ x} u_1(x),
\end{equation*}
so that the solution to $(\ref{Main.Eq.1})$ can be expressed by
\begin{equation*}\label{2.31}
    u(t,x) = \varepsilon K_0(t,x) \ast_x u_0(x) + \varepsilon K_1(t,x) \ast_x u_1(x) + \int_0^t K_1(t-\tau,x) \ast_x |u(\tau,x)|^p d\tau,
\end{equation*}
thanks to Duhamel's principle.
Here 
\begin{equation*}
\widehat{K_0}(t,\xi) = \frak{F}_{x \to \xi}(K_0(t,x)) = \frac{\lambda_{1} e^{\lambda_{2}t}-\lambda_{2}e^{\lambda_{1}t}}{\lambda_{1} - \lambda_{2}}
\end{equation*}
and
\begin{equation*}
\widehat{K_1}(t,\xi) = \frak{F}_{x\to\xi}(K_1(t,x))=  \frac{e^{\lambda_{1}t}-e^{\lambda_{2}t}}{\lambda_{1} - \lambda_{2}}.
\end{equation*}
The characteristic roots $\lambda_{1,2} =\lambda_{1,2}(|\xi|)$ are given by
\begin{equation*}
\lambda_{1,2}(|\xi|) =
\begin{cases}
\vspace{0.2cm}
    \f{1}{2}\left(-|\xi|^{2\sigma_1} - |\xi|^{2\sigma_2} \pm \sqrt{(|\xi|^{2\sigma_1} + |\xi|^{2\sigma_2})^{2} - 4 |\xi|^{2\sigma}}\right) & \text { if } |\xi| \in \mathbb{R}_+ \setminus \Omega, \\
     \f{1}{2}\left(-|\xi|^{2\sigma_1} - |\xi|^{2\sigma_2} \pm i \sqrt{4 |\xi|^{2\sigma}-(|\xi|^{2\sigma_1} + |\xi|^{2\sigma_2})^{2} }\right) & \text { if } |\xi| \in \Omega,  
\end{cases}
\end{equation*}
where $\Omega = \left\{r \in \mathbb{R}_+:  r^{2\sigma_1} +  r^{2\sigma_2} < 2 r^{\sigma} \right\}$. Because of $0 \leq \sigma_1 <\sigma/2 < \sigma_2 \leq \sigma$, there exists a sufficiently small constant $\varepsilon^*>0$ such that
\begin{equation*}
    (0, \varepsilon^*)\cup \left(\frac{1}{\varepsilon^*},\ity\right) \subset \Omega.
\end{equation*}
Then, taking account of the cases of small and large frequencies separately we conclude that
\begin{align}
     &\lambda_{1} \sim -|\xi|^{2(\sigma-\sigma_1)},\quad 
    \lambda_{2} \sim -|\xi|^{2\sigma_1}, \quad \lambda_1-\lambda_2 \sim |\xi|^{2\sigma_1}\quad \text{ for } |\xi| \leq \varepsilon^*, \label{Re1}\\
    &\lambda_{1} \sim -|\xi|^{2(\sigma-\sigma_2)} ,\quad \lambda_{2} \sim -|\xi|^{2\sigma_2}, \quad \lambda_1-\lambda_2 \sim |\xi|^{2\sigma_2} \quad \text{ for }|\xi| \geq \frac{1}{\varepsilon^*}. \label{Re2}
\end{align} 
Let $\chi_k= \chi_k(r)$ with $k\in\{\rm L,H\}$ be smooth cut-off functions having the following properties:
\begin{align*}
&\chi_{\rm L}(r)=
\begin{cases}
1 &\quad \text{ if }r\le \varepsilon^*/2 \\
0 &\quad \text{ if }r\ge \varepsilon^*
\end{cases}
\quad \text{ and } \qquad
\chi_{\rm H}(r)= 1 -\chi_{\rm L}(r).
\end{align*}
It is obvious to see that $\chi_{\rm H}(r)= 1$ if $r \geq \varepsilon^*$ and $\chi_{\rm H}(r)= 0$ if $r \le \varepsilon^*/2$. Next, let $ T\in (0,\infty]$ and $m \in [1,2]$. We define the following function space:
\begin{align*}
    X(T) :=  \mathcal{C}\left([0,T), H^{\sigma}\right)
\end{align*}
with the norm
\begin{align*}
    &\|\varphi\|_{X(T)} := \sup _{t \in [0,T)} \bigg\{(1+t)^{\frac{n}{2(\sigma-\sigma_1)}(\frac{1}{m}-\frac{1}{2})-\frac{\sigma_1}{\sigma-\sigma_1}}\|\varphi(t,\cdot)\|_{L^2}\\
    &\hspace{4cm}+ (1+t)^{\frac{n}{2(\sigma-\sigma_1)}(\frac{1}{m}-\frac{1}{2})+\frac{\sigma-2\sigma_1}{2(\sigma-\sigma_1)}} \||D|^{\sigma} \varphi(t,\cdot)\|_{L^2}  \bigg\}.
\end{align*}
Let $M > 0$, we consider the closed ball 
\begin{align*}
    X(T,M) := \left\{u \in X(T)\,:\, \|u\|_{X(T)} \leq M \right\}
\end{align*}
and the function space
\begin{align*}
    Z(T) := \mathcal{C}([0,T),L^2)
\end{align*}
with the norm
\begin{align*}
    \|\varphi\|_{Z(T)} := \sup _{t \in [0,T)} \left\{(1+t)^{-\frac{\sigma_1}{\sigma-\sigma_1}}\|\varphi(t,\cdot)\|_{L^2}\right\}.
\end{align*}
Furthermore, we also define the function space
\begin{align*}
    Y(T) := \mathcal{C}\left([0,T), L^2 \cap L^{m_1}\right),
\end{align*}
with the norm
\begin{align*}
    \|\varphi\|_{Y(T)} &:= \sup _{t \in [0,T)}\bigg\{(1+t)^{\frac{np}{2(\sigma-\sigma_1)}(\frac{1}{m}-\frac{1}{2p}) - \frac{p\sigma_1}{\sigma-\sigma_1} }\| \varphi(t,\cdot)\|_{L^2} \\
    &\hspace{4cm}+  \sup _{\beta \in [m_1, 2]}(1+t)^{\frac{n}{2(\sigma-\sigma_1)}(\frac{p}{m}-\frac{1}{\beta})-\frac{p\sigma_1}{\sigma-\sigma_1}} \|\varphi(t,\cdot)\|_{L^{\beta}}\bigg\},
\end{align*}
where we choose $m_1 := \max\left\{1, 2/p\right\}.$ \medskip

Following the same approach as the proof of Theorem 2.1 in \cite{DaoDuongNguyen2024} with some minor modifications we may conclude the auxiliary estimates as follows.

\begin{lemma}\label{Proposition3.1}
    Let $n \geq 1$, $s \geq 0$ and $0 \leq \sigma_1 < \sigma/2 < \sigma_2 \leq \sigma $. The parameter $m$ satisfies the condition (\ref{condition_Local_1.1.1}) or $m=1$. The following $(L^m \cap L^2) -L^2$ estimates hold for $ j \in \{0,1\}$:
    \begin{align*}
        \left\|\partial_t^j|D|^s K_1(t,x)*_{x}\varphi(x)\right\|_{L^2} &\lesssim (1+t)^{-\frac{n}{2(\sigma-\sigma_1)}(\frac{1}{m}-\frac{1}{2})+\frac{\sigma_1}{\sigma-\sigma_1}-\frac{s}{2(\sigma-\sigma_1)}-j} \|\varphi\|_{L^m \cap H^{s+2(j-1)\sigma_2}}
    \end{align*}
    and the $L^2-L^2$ estimates
    \begin{align*}
       \left\||D|^s K_1(t,x)*_{x}\varphi(x)\right\|_{L^2} &\lesssim
       \begin{cases}
           (1+t)^{-\frac{s}{2(\sigma-\sigma_1)}+1} \|\varphi\|_{L^2} &\text{ if } s < 2\sigma_1,\\
           (1+t)^{-\frac{s}{2(\sigma-\sigma_1)}+\frac{\sigma_1}{\sigma-\sigma_1}} \|\varphi\|_{H^{s-2\sigma_2}} &\text{ if } s \geq 2\sigma_1,
       \end{cases}\\
       \left\|\partial_t|D|^s K_1(t,x) \ast_x \varphi(x) \right\|_{L^2} &\lesssim (1+t)^{-\frac{s}{2(\sigma-\sigma_1)}} \|\varphi\|_{H^s}.
    \end{align*}
\end{lemma}

Next, to prove our main theorems, the following tools from Harmonic Analysis come into play.
\begin{lemma}[Fractional Gagliardo-Nirenberg inequality, see \cite{Ozawa}] \label{FractionalG-N}
Let $1<q,\, q_1,\, q_2<\infty$, $a >0$ and $\theta\in [0,a)$. Then, it holds
$$ \|u\|_{\dot{H}^{\theta}_q}\lesssim \|u\|_{L^{q_1}}^{1-\eta_{\theta,a}}\,\, \|u\|_{\dot{H}^a_{q_2}}^{\eta_{\theta,a}}, $$
where
$$ \eta_{\theta,a}(q,q_1,q_2)=\frac{\frac{1}{q_1}-\frac{1}{q}+\frac{\theta}{n}}{\frac{1}{q_1}-\frac{1}{q_2}+\frac{a}{n}}\quad \text{ and }\quad \frac{\theta}{a}\leq \eta_{\theta, a}\leq 1. $$
\end{lemma}
Additionally, we need to use the following auxiliary ingredients.
\begin{proposition}\label{lemma3.1}
    Under the assumptions of Theorem \ref{Local-existence}, we have
     \begin{align}
        &\left\|\int_0^t K_1(t-\tau, x)*_{x}\varphi(\tau,x) d\tau \right\|_{X(T)}\notag\\
        &\hspace{1cm}\lesssim 
         \begin{cases}
         \vspace{0.2cm}\|\varphi\|_{Y(T)} &\text{ if } p \geq 1 + \dfrac{2m\sigma}{n-2m\sigma_1},\\
         (1+T)^{1-\frac{n}{2m(\sigma-\sigma_1)}(p-1) + \frac{p\sigma_1}{\sigma-\sigma_1}} \|\varphi\|_{Y(T)} &\text{ if } p < 1 + \dfrac{2m\sigma}{n- 2m\sigma_1} ,
         \end{cases}\label{estimate3.1.1}
    \end{align}
    for any $T > 0$. Moreover, we also obtain
   \begin{align}
        \left\|\int_0^t K_1(t-\tau, x) *_{x} \varphi(\tau, x) d\tau \right\|_{X(T)} \lesssim \int_0^T \|\varphi\|_{Y(\tau)} d \tau, \label{estimate3.1.2}
    \end{align}
    for all $0 < T \leq 1$.
\end{proposition}
\begin{proof}
    To prove Proposition \ref{lemma3.1}, we will apply Lemma \ref{Proposition3.1} to $s = \sigma$. Indeed, using the following notations:
    \begin{align*}
        I_1(t) := &(1+t)^{\frac{n}{2(\sigma-\sigma_1)}(\frac{1}{m}-\frac{1}{2})-\frac{\sigma_1}{\sigma-\sigma_1}}\left\|\int_0^{t/2} K_1(t-\tau, x)*_{x} \varphi(\tau,x)d\tau\right\|_{L^2} \\
        &\qquad\quad+(1+t)^{\frac{n}{2(\sigma-\sigma_1)}(\frac{1}{m}-\frac{1}{2})+\frac{\sigma-2\sigma_1}{2(\sigma-\sigma_1)}}\left\|\int_0^{t/2} |D|^{\sigma} K_1(t-\tau, x)*_{x} \varphi(\tau,x)d\tau\right\|_{L^2}
    \end{align*}
    and 
    \begin{align*}
         I_2(t) := &(1+t)^{\frac{n}{2(\sigma-\sigma_1)}(\frac{1}{m}-\frac{1}{2})-\frac{\sigma_1}{\sigma-\sigma_1}}\left\|\int_{t/2}^t K_1(t-\tau, x)*_{x} \varphi(\tau,x)d\tau\right\|_{L^2} \\
        &\qquad+(1+t)^{\frac{n}{2(\sigma-\sigma_1)}(\frac{1}{m}-\frac{1}{2})+\frac{\sigma-2\sigma_1}{2(\sigma-\sigma_1)}}\left\|\int_{t/2}^t |D|^{\sigma} K_1(t-\tau, x)*_{x} \varphi(\tau,x)d\tau\right\|_{L^2},
    \end{align*}
    we can see that
    \begin{align*}
        \left\|\int_0^t K_1(t-\tau, x)*_{x}\varphi(\tau,x) d\tau \right\|_{X(T)} \leq \sup_{t \in [0, T)} \left\{I_1(t) + I_2(t)\right\}
    \end{align*}
    for any $T>0$.
Applying the $L^{m_1}-L^2$ estimates of Lemma \ref{Proposition3.1} for the term $I_1(t)$ with $j = 0$ and $s = \sigma < 2\sigma_2$, we arrive at
\begin{align*}
    I_1(t)&\lesssim (1+t)^{\frac{n}{2(\sigma-\sigma_1)}(\frac{1}{m}-\frac{1}{2})-\frac{\sigma_1}{\sigma-\sigma_1}} \\
    &\hspace{1cm}\times\int_0^{t/2}(1+t-\tau)^{-\frac{n}{2(\sigma-\sigma_1)}(\frac{1}{m_1}-\frac{1}{2})+\frac{\sigma_1}{\sigma-\sigma_1}} \|\varphi(\tau,\cdot)\|_{L^{m_1}\cap L^2}d\tau\\
        &\qquad
        + (1+t)^{\frac{n}{2(\sigma-\sigma_1)}(\frac{1}{m}-\frac{1}{2})+\frac{\sigma-2\sigma_1}{2(\sigma-\sigma_1)}} \\
        &\hspace{1cm}\times\int_0^{t/2}(1+t-\tau)^{-\frac{n}{2(\sigma-\sigma_1)}(\frac{1}{m_1}-\frac{1}{2})-\frac{\sigma-2\sigma_1}{2(\sigma-\sigma_1)}} \|\varphi(\tau,\cdot)\|_{L^{m_1} \cap L^2}d\tau\\
        &=: I_{1,1}(t) + I_{1,2}(t).
\end{align*}
 From the norm definition of the space $Y(T)$, we have the following estimates:
\begin{align}
    \|\varphi(\tau,\cdot)\|_{L^{m_1}} &\lesssim (1+\tau)^{-\frac{n}{2(\sigma-\sigma_1)}(\frac{p}{m}-\frac{1}{m_1})+ \frac{p\sigma_1}{\sigma-\sigma_1}}\|\varphi\|_{Y(\tau)},\label{E1}\\
     \|\varphi(\tau,\cdot)\|_{L^2} &\lesssim (1+\tau)^{-\frac{n}{2(\sigma-\sigma_1)}(\frac{p}{m}-\frac{1}{2})+ \frac{p\sigma_1}{\sigma-\sigma_1}}\|\varphi\|_{Y(\tau)}.  \label{E2}
\end{align}
For this reason and $m > m_1$, one has
\begin{align*}
    I_{1,1}(t) &\lesssim (1+t)^{-\frac{n}{2(\sigma-\sigma_1)}(\frac{1}{m_1}-\frac{1}{m})} \int_0^{t/2} (1+\tau)^{-\frac{n}{2(\sigma-\sigma_1)}(\frac{p}{m}-\frac{1}{m_1})+\frac{p\sigma_1}{\sigma-\sigma_1}} \|\varphi\|_{Y(\tau)} d\tau\\
    &\lesssim 
    \begin{cases}
         \vspace{0.2cm}\|\varphi\|_{Y(T)} &\text{ if } p \geq 1 + \dfrac{2m\sigma}{n-2m\sigma_1},\\
         (1+T)^{1-\frac{n}{2m(\sigma-\sigma_1)}(p-1) + \frac{p\sigma_1}{\sigma-\sigma_1}} \|\varphi\|_{Y(T)} &\text{ if } p < 1 + \dfrac{2m\sigma}{n- 2m\sigma_1} ,
         \end{cases}
\end{align*}
and 
\begin{align*}
    I_{1,2}(t) &\lesssim (1+t)^{-\frac{n}{2(\sigma-\sigma_1)}(\frac{1}{m_1}-\frac{1}{m})} \int_0^{t/2} (1+\tau)^{-\frac{n}{2(\sigma-\sigma_1)}(\frac{p}{m}-\frac{1}{m_1})+\frac{p\sigma_1}{\sigma-\sigma_1}} \|\varphi\|_{Y(\tau)} d\tau\\
    &\lesssim 
    \begin{cases}
         \|\varphi\|_{Y(T)} &\text{ if } p \geq 1 + \dfrac{2m\sigma}{n-2m\sigma_1},\\
         (1+T)^{1-\frac{n}{2m(\sigma-\sigma_1)}(p-1) + \frac{p\sigma_1}{\sigma-\sigma_1}} \|\varphi\|_{Y(T)} &\text{ if } p < 1 + \dfrac{2m\sigma}{n- 2m\sigma_1},
         \end{cases} 
\end{align*}
for all $t \in [0, T)$, provided that
\begin{align*}
    1 - \frac{n}{2m(\sigma-\sigma_1)}(p-1) + \frac{p\sigma_1}{\sigma-\sigma_1} \leq 0, &\text{ i.e. } p \geq 1 +\frac{2m\sigma}{n-2m\sigma_1}.
\end{align*}
As a consequence, it follows that
\begin{align*}
    I_1(t) \lesssim \begin{cases}
         \vspace{0.2cm}\|\varphi\|_{Y(T)} &\text{ if } p \geq 1 + \dfrac{2m\sigma}{n-2m\sigma_1},\\
         (1+T)^{1-\frac{n}{2m(\sigma-\sigma_1)}(p-1) + \frac{p\sigma_1}{\sigma-\sigma_1}} \|\varphi\|_{Y(T)} &\text{ if } p < 1 + \dfrac{2m\sigma}{n- 2m\sigma_1} ,
         \end{cases}
\end{align*}
for all $t \in [0, T)$.
Next we will estimate the term $I_2(t)$. Noting that $m \in (m_1, 2]$, one finds
\begin{align}
    \|\varphi(\tau,\cdot)\|_{L^{m}} &\lesssim (1+\tau)^{-\frac{n}{2m(\sigma-\sigma_1)}(p-1)+ \frac{p\sigma_1}{\sigma-\sigma_1}}\|\varphi\|_{Y(\tau)}. \label{E3} 
\end{align}
Using (\ref{E2}) and (\ref{E3}), we can proceed with the following steps:
\begin{align*}
    I_2(t) \lesssim& (1+t)^{\frac{n}{2(\sigma-\sigma_1)}(\frac{1}{m}-\frac{1}{2})-\frac{\sigma_1}{\sigma-\sigma_1}} \int_{t/2}^t (1+t-\tau)^{-\frac{n}{2(\sigma-\sigma_1)}(\frac{1}{m}-\frac{1}{2})+\frac{\sigma_1}{\sigma-\sigma_1}} \|\varphi(\tau,\cdot)\|_{L^m \cap L^2} d\tau\\
    &+ (1+t)^{\frac{n}{2(\sigma-\sigma_1)}(\frac{1}{m}-\frac{1}{2})+\frac{\sigma-2\sigma_1}{2(\sigma-\sigma_1)}} \int_{t/2}^t (1+t-\tau)^{-\frac{\sigma-2\sigma_1}{2(\sigma-\sigma_1)}} \|\varphi(\tau,\cdot)\|_{L^2} d\tau\\
    \lesssim& (1+t)^{\frac{n}{2(\sigma-\sigma_1)}(\frac{1}{m}-\frac{1}{2})-\frac{\sigma_1}{\sigma-\sigma_1}} \\
    &\hspace{1cm}\int_{t/2}^t (1+t-\tau)^{-\frac{n}{2(\sigma-\sigma_2)}(\frac{1}{m}-\frac{1}{2})+\frac{\sigma_1}{\sigma-\sigma_1}} (1+\tau)^{-\frac{n}{2m(\sigma-\sigma_1)}(p-1) + \frac{p\sigma_1}{\sigma-\sigma_1}} \|\varphi\|_{Y(\tau)}d\tau\\
    &+ (1+t)^{\frac{n}{2(\sigma-\sigma_1)}(\frac{1}{m}-\frac{1}{2})+\frac{\sigma-2\sigma_1}{2(\sigma-\sigma_1)}}\\
    &\hspace{1cm}\int_{t/2}^t (1+t-\tau)^{-\frac{\sigma-2\sigma_1}{2(\sigma-\sigma_1)}} (1+\tau)^{-\frac{n}{2(\sigma-\sigma_1)}(\frac{p}{m}-\frac{1}{2})+\frac{p\sigma_1}{\sigma-\sigma_1}} \|\varphi\|_{Y(\tau)} d\tau \\
    \lesssim&
     \begin{cases}
         \vspace{0.2cm}\|\varphi\|_{Y(T)} &\text{ if } p \geq 1 + \dfrac{2m\sigma}{n-2m\sigma_1},\\
         (1+T)^{1-\frac{n}{2m(\sigma-\sigma_1)}(p-1) + \frac{p\sigma_1}{\sigma-\sigma_1}} \|\varphi\|_{Y(T)} &\text{ if } p < 1 + \dfrac{2m\sigma}{n- 2m\sigma_1},
         \end{cases}
\end{align*}
for all $t \in [0, T)$, where we applied Lemma \ref{Proposition3.1} in the first estimate and the fact that
\begin{align*}
   0 \leq  \frac{n}{2(\sigma-\sigma_1)}\left(\frac{1}{m}-\frac{1}{2}\right)-\frac{\sigma_1}{\sigma-\sigma_1} < 1 &\text{ together with } 0 < \frac{\sigma-2\sigma_1}{2(\sigma-\sigma_1)} < 1
\end{align*}
for the last estimate. From this, we may conclude the estimate (\ref{estimate3.1.1}). The estimate (\ref{estimate3.1.2}) can be deduced from the proof steps above. Thus, Proposition \ref{lemma3.1} has been proved.
\end{proof}

\begin{proposition}\label{lemma3.2}
    Under the assumptions of Theorem \ref{Local-existence}, we have
    \begin{align*}
        \| |u|^p\|_{Y(T)} \lesssim \|u\|_{X(T)}^p
    \end{align*}
    for all $T \in (0,\infty)$.
\end{proposition}
\begin{proof}
    Applying Lemma \ref{FractionalG-N} for any $t \in [0, T)$ and $ \beta \in [m_1, 2]$, we obtain
    \begin{align*}
        \||u(t,\cdot)|^p\|_{L^{\beta}} = \|u(t,\cdot)\|_{L^{p\beta}}^p &\lesssim \|u(t, \cdot)\|_{L^2}^{p(1-\theta_1)} \||D|^{\sigma} u(t,\cdot)\|_{L^2}^{p \theta_1}\\ 
        &\lesssim (1+t)^{-\frac{n}{2(\sigma-\sigma_1)}(\frac{p}{m}-\frac{1}{\beta})+\frac{p\sigma_1}{\sigma-\sigma_1}}\|u\|_{X(T)}^p,\\
    \||u(t,\cdot)|^p\|_{L^2} = \|u(t,\cdot)\|_{L^{2p}}^p &\lesssim \|u(t, \cdot)\|_{L^2}^{p(1-\theta_2)} \||D|^{\sigma} u(t,\cdot)\|_{L^2}^{p \theta_2}\\
    &\lesssim (1+t)^{-\frac{n}{2(\sigma-\sigma_1)}(\frac{p}{m}-\frac{1}{2})+\frac{p\sigma_1}{\sigma-\sigma_1}}\|u\|_{X(T)}^p,
    \end{align*}
    where
    $$ \theta_1 = \frac{n}{\sigma}\left(\frac{1}{2}-\frac{1}{p\beta}\right) \in [0,1] \text{ and } \theta_2 = \frac{n}{\sigma}\left(\frac{1}{2}-\frac{1}{2p}\right) \in [0,1], $$
    that is, 
    $$p \in \left(1, \frac{n}{n-2\sigma}\right] \text{ if } n > 2\sigma \text{ or } p \in (1, \infty) \text{ if }n \leq 2\sigma. $$
    Thus, Proposition \ref{lemma3.2} has been proved.
\end{proof}
\subsection{Proof of Theorem \ref{Local-existence}}
Let $0 < T \leq 1$. From Lemma \ref{Proposition3.1} we derive
\begin{align}
    \left\| K_0(t,x)*_{x}u_0(x)+ K_1(t,x)*_{x}u_1(x)\right\|_{X(T)} \leq C_1\|(u_0, u_1)\|_{\mathcal{D}_{\sigma}^m} =: \frac{M}{2} \label{Constant1}
\end{align}
with some constant $C_1 > 0$, which is independent of $T$, and the following relation:
\begin{align*}
    \widehat{K_0}(t,\xi) = (|\xi|^{2\sigma_2}+|\xi|^{2\sigma_1})\widehat{K_1}(t,\xi) +\partial_t \widehat{K_1}(t,\xi).
\end{align*}
We consider the function space $X(T, M\varepsilon)$ with $u \in X(T, M\varepsilon)$ and define the mapping
\begin{align}
    \mathcal{N}[u](t,x) := \varepsilon K_0(t,x)*_{x}u_0(x) + \varepsilon K_1(t,x)*_{x}u_1(x)+ \int_0^t K_1(t-\tau,x)*_{x}|u(\tau,x)|^p d\tau. \label{map1}
\end{align}
From the estimate (\ref{estimate3.1.2}) and Proposition \ref{lemma3.2} we arrive at
\begin{align*}
    \left\|\int_0^t K_1(t-\tau,x)*_{x}|u(\tau,x)|^p d\tau\right\|_{X(T)} &\lesssim \int_0^T \||u|^p\|_{Y(\tau)} d\tau \lesssim T M^p \varepsilon^p.
\end{align*}
Thus, it follows that
\begin{equation}\label{estimate1.2.1}
    \|\mathcal{N}[u]\|_{X(T)} \leq \frac{M\varepsilon}{2} + C_2 T M^p\varepsilon^p .
\end{equation}
For $u, v \in X(T,M\varepsilon)$, we intend to prove the following inequality:
\begin{equation}\label{estimate1.2.2}
    \|\mathcal{N}[u]-\mathcal{N}[v]\|_{Z(T)} \leq C_3 T M^{p-1}  \varepsilon^{p-1}  \|u-v\|_{Z(T)}.
\end{equation}
To get (\ref{estimate1.2.2}), we define the constants $q$ and $\gamma$ as follows:
\begin{align*}
   \frac{1}{q} =
    \begin{cases}
        \delta &\text{ if } n \leq 2\sigma \\
        \displaystyle\frac{(p-1)(n-2\sigma)}{2n} &\text{ if } n > 2\sigma
    \end{cases} 
    \quad \text{ and }\quad
    \frac{1}{\gamma} = \frac{1}{q}+\frac{1}{2}
\end{align*}
with a sufficiently small constant $\delta > 0$. We see that condition $p \leq n/(n-2\sigma)$ if $n > 2\sigma$ implies $q > 2$ and $\gamma \in (1, 2)$. Due to $0 < t < T \leq 1$, combining with Lemma \ref{Proposition3.1} one has
\begin{align*}
    &\left\|\int_0^t K_1(t-\tau,x)*_{x} \big(|u(\tau,x)|^p-|v(\tau,x)|^p\big)d\tau\right\|_{L^2}\\
    &\quad\lesssim \int_0^t (1+t-\tau)^{-\frac{n}{2(\sigma-\sigma_1)}(\frac{1}{\gamma}-\frac{1}{2})+\frac{\sigma_1}{\sigma-\sigma_1}} \||u(\tau,\cdot)|^p-|v(\tau,\cdot)|^p\|_{L^{\gamma} \cap H^{-2\sigma_2}}d\tau\\
    &\quad\lesssim \int_0^t (1+t-\tau)^{-\frac{n}{2(\sigma-\sigma_1)}(\frac{1}{\gamma}-\frac{1}{2})+\frac{\sigma_1}{\sigma-\sigma_1}} \||u(\tau,\cdot)|^p-|v(\tau,\cdot)|^p\|_{L^{\gamma} \cap H^{-\sigma}}d\tau\\
    &\quad\lesssim \int_0^t \||u(\tau,\cdot)|^p-|v(\tau,\cdot)|^p\|_{L^{\gamma} \cap H^{-\sigma}}d\tau .
    \end{align*}
    To estimate the norm in the previous integral, we use the embedding 
    \begin{align}
        \| \psi\|_{L^2} \lesssim \| \psi\|_{H_{\gamma}^{\sigma}}, \label{embedding}
    \end{align}
    with $1/2 + \sigma/n \geq 1/\gamma$. As a result, we achieve
    \begin{align*}
        \left\|\mathcal{N}[u]- \mathcal{N}[v]\right\|_{Z(T)} &\lesssim \sup_{t \in [0, T)}\int_0^t \||u(\tau,\cdot)|^p-|v(\tau,\cdot)|^p\|_{L^{\gamma}} d\tau\\
        &\lesssim \sup_{t \in [0, T)}\int_0^t  \|u(\tau,\cdot)-v(\tau,\cdot)\|_{L^2}\left(\|u(\tau,\cdot)\|_{L^{q(p-1)}}^{p-1}+\|v(\tau,\cdot)\|_{L^{q(p-1)}}^{p-1}\right) d\tau\\
        &\lesssim T \|u-v\|_{Z(T)} \left(\|u\|_{X(T)}^{p-1} + \|v\|_{X(T)}^{p-1}\right). 
    \end{align*}
    In the second estimate, we have used H\"older's inequality with $1/\gamma = 1/2 +1/q$. Meanwhile, in the last estimate, we have utilized Lemma \ref{FractionalG-N} for the norms
    $$\|u(\tau,\cdot)\|_{L^{q(p-1)}}, \, \|v(\tau,\cdot)\|_{L^{q(p-1)}} $$ 
    and the estimate (\ref{estimate3.1.2}). From this, we can conclude the estimate (\ref{estimate1.2.2}). Finally, by fixing $T=T(C_2, C_3,\varepsilon)$ as a sufficiently small positive constant, we can see that $\mathcal{N}$ is a contraction mapping in $X(T,M\varepsilon)$ with the metric $Z(T)$. From the application of Banach's fixed point theorem, we may conclude Theorem \ref{Local-existence}.
\subsection{Proof of Theorem \ref{Global-Existance}} For any 
$T > 0$, we define the constant $M$ in (\ref{Constant1}) and consider the mapping $\mathcal{N}$ in (\ref{map1}) on $X(T, M\varepsilon)$.
For any $u \in X(T,M\varepsilon)$, from Propositions \ref{lemma3.1} and \ref{lemma3.2} we arrive at
\begin{equation}\label{estimate1.3.1}
    \|\mathcal{N}[u]\|_{X(T)} \leq \frac{M\varepsilon}{2} + C_1 M^p\varepsilon^p.
\end{equation}
For $u, v \in X(T,M\varepsilon)$, we are going to prove the following inequality:
\begin{equation}\label{estimate1.3.2}
    \|\mathcal{N}[u]-\mathcal{N}[v]\|_{Z(T)} \leq C_2M^{p-1}\varepsilon^{p-1} \|u-v\|_{Z(T)}.
\end{equation}
To obtain (\ref{estimate1.3.2}), we define the constants $q$ and $\gamma$ as in the proof of Theorem \ref{Local-existence}.
 The fact is that
\begin{align*}
   &\left\|\int_0^t K_1(t-\tau,x)*_{x} \big(|u(\tau,x)|^p-|v(\tau,x)|^p\big)d\tau\right\|_{L^2}\\
    &\quad\lesssim \int_0^t (1+t-\tau)^{-\frac{n}{2(\sigma-\sigma_1)}(\frac{1}{\gamma}-\frac{1}{2})+\frac{\sigma_1}{\sigma-\sigma_1}} \||u(\tau,\cdot)|^p-|v(\tau,\cdot)|^p\|_{L^{\gamma} \cap H^{-2\sigma_2}}d\tau\\
    &\quad\lesssim \int_0^t (1+t-\tau)^{-\frac{n}{2(\sigma-\sigma_1)}(\frac{1}{\gamma}-\frac{1}{2})+\frac{\sigma_1}{\sigma-\sigma_1}} \||u(\tau,\cdot)|^p-|v(\tau,\cdot)|^p\|_{L^{\gamma} \cap H^{-\sigma}}d\tau.
    \end{align*}
    To estimate the norm in the last integral, we use the embedding (\ref{embedding}) again to derive
    \begin{align*}
        \||u(\tau,\cdot)|^p-|v(\tau,\cdot)|^p\|_{H^{-\sigma}} \lesssim \||u(\tau,\cdot)|^p-|v(\tau,\cdot)|^p\|_{L^{\gamma}}
    \end{align*}
    as long as $1/2 \geq 1/\gamma - \sigma/n$, which holds from the definition of $\gamma$. Moreover, using H\"older's inequality with $1/\gamma = 1/2 + 1/q$, we get
    \begin{align*}
        &\||u(\tau,\cdot)|^p-|v(\tau,\cdot)|^p\|_{L^{\gamma}} \\
        &\quad\lesssim \|u(\tau,\cdot)-v(\tau,\cdot)\|_{L^2}\left(\|u(\tau,\cdot)\|_{L^{q(p-1)}}^{p-1}+\|v(\tau,\cdot)\|_{L^{q(p-1)}}^{p-1}\right)\\
        &\quad\lesssim (1+\tau)^{\frac{\sigma_1}{\sigma-\sigma_1}} \|u-v\|_{Z(T)} (1+\tau)^{ -\frac{n(p-1)}{2(\sigma-\sigma_1)}(\frac{1}{m}-\frac{1}{q(p-1)})+\frac{\sigma_1}{\sigma-\sigma_1}(p-1)} \left(\|u\|_{X(T)}^{p-1} +\|v\|_{X(T)}^{p-1}\right)\\
        &\quad\lesssim (1+\tau)^{-\frac{n(p-1)}{2(\sigma-\sigma_1)}(\frac{1}{m}-\frac{1}{q(p-1)}) + \frac{p\sigma_1}{\sigma-\sigma_1}} \|u-v\|_{Z(T)}\left(\|u\|_{X(T)}^{p-1} +\|v\|_{X(T)}^{p-1}\right).
    \end{align*}
    In the second estimate, we have used Lemma \ref{FractionalG-N} for the norms $\|u(\tau,\cdot)\|_{L^{q(p-1)}}$ and $\|v(\tau,\cdot)\|_{L^{q(p-1)}}$ under the condition
    $$\frac{n}{\sigma}\left(\frac{1}{2}-\frac{1}{q(p-1)}\right) \in [0, 1]. $$
    For this reason, we obtain
    \begin{align*}
    &\|\mathcal{N}[u]-\mathcal{N}[v]\|_{Z(T)}\\
    &\quad= \sup_{t \in[0, T)} \left\{ (1+t)^{-\frac{\sigma_1}{\sigma-\sigma_1}} \left\|\int_0^t K_1(t-\tau,x)*_{x} \big(|u(\tau,x)|^p-|v(\tau,x)|^p\big)d\tau\right\|_{L^2}\right\}\\
    &\quad\lesssim  \|u-v\|_{Z(T)}\left(\|u\|_{X(T)}^{p-1}+\|v\|_{X(T)}^{p-1}\right)\\
    &\hspace{0.5cm}\times \sup_{t \in [0, T)} \left\{(1+t)^{-\frac{\sigma_1}{\sigma-\sigma_1}}\int_0^t (1+t-\tau)^{-\frac{n}{2(\sigma-\sigma_1)}(\frac{1}{\gamma}-\frac{1}{2})+\frac{\sigma_1}{\sigma-\sigma_1}} (1+\tau)^{-\frac{n(p-1)}{2(\sigma-\sigma_1)}(\frac{1}{m}-\frac{1}{q(p-1)})+\frac{p\sigma_1}{\sigma-\sigma_1}} d\tau\right\}\\
    &\quad\lesssim M^{p-1} \varepsilon^{p-1} \|u-v\|_{Z(T)}\\
    &\hspace{0.5cm}\times \sup_{t \in [0, T)} \left\{(1+t)^{-\frac{\sigma_1}{\sigma-\sigma_1}}\int_0^t (1+t-\tau)^{-\frac{n}{2(\sigma-\sigma_1)}(\frac{1}{\gamma}-\frac{1}{2})+\frac{\sigma_1}{\sigma-\sigma_1}} (1+\tau)^{-\frac{n(p-1)}{2(\sigma-\sigma_1)}(\frac{1}{m}-\frac{1}{q(p-1)})+\frac{p\sigma_1}{\sigma-\sigma_1}} d\tau\right\}.
\end{align*}
 Now, we will prove the following inequality:
\begin{align*}
    (1+t)^{-\frac{\sigma_1}{\sigma-\sigma_1}}\int_0^t (1+t-\tau)^{-\frac{n}{2(\sigma-\sigma_1)}(\frac{1}{\gamma}-\frac{1}{2})+\frac{\sigma_1}{\sigma-\sigma_1}} (1+\tau)^{-\frac{n(p-1)}{2(\sigma-\sigma_1)}(\frac{1}{m}-\frac{1}{q(p-1)})+ \frac{p\sigma_1}{\sigma-\sigma_1}} d\tau \lesssim 1,
\end{align*}
for all $t > 0$. We divide the left side into two parts
\begin{align*}
    J_1(t):=& (1+t)^{-\frac{\sigma_1}{\sigma-\sigma_1}}\int_0^{t/2} (1+t-\tau)^{-\frac{n}{2(\sigma-\sigma_1)}(\frac{1}{\gamma}-\frac{1}{2})+\frac{\sigma_1}{\sigma-\sigma_1}} (1+\tau)^{-\frac{n(p-1)}{2(\sigma-\sigma_1)}(\frac{1}{m}-\frac{1}{q(p-1)})+\frac{p\sigma_1}{\sigma-\sigma_1}} d\tau,\\
    J_2(t):=& (1+t)^{-\frac{\sigma_1}{\sigma-\sigma_1}}\int_{t/2}^t (1+t-\tau)^{-\frac{n}{2(\sigma-\sigma_1)}(\frac{1}{\gamma}-\frac{1}{2})+\frac{\sigma_1}{\sigma-\sigma_1}} (1+\tau)^{-\frac{n(p-1)}{2(\sigma-\sigma_1)}(\frac{1}{m}-\frac{1}{q(p-1)})+\frac{p\sigma_1}{\sigma-\sigma_1}} d\tau.
\end{align*}
The term $J_1(t)$ is estimated as
\begin{align*}
    J_1(t) \lesssim (1+t)^{-\frac{n}{2(\sigma-\sigma_1)}(\frac{1}{\gamma}-\frac{1}{2})}\int_0^{t/2} (1+\tau)^{-\frac{n(p-1)}{2(\sigma-\sigma_1)}(\frac{1}{m}-\frac{1}{q(p-1)})+ \frac{p\sigma_1}{\sigma-\sigma_1}} d\tau.
\end{align*}
We see that if
$$-\frac{n(p-1)}{2(\sigma-\sigma_1)}\left(\frac{1}{m}-\frac{1}{q(p-1)}\right) + \frac{p\sigma_1}{\sigma-\sigma_1}\leq -1, $$
then $ J_1(t) \lesssim 1$ due to $\gamma \in (1,2)$. Otherwise, we also compute 
\begin{align*}
    J_1(t) \lesssim (1+t)^{1-\frac{n}{2(\sigma-\sigma_1)}(\frac{p-1}{m}-\frac{1}{q}+\frac{1}{\gamma}-\frac{1}{2})+\frac{p\sigma_1}{\sigma-\sigma_1}}= (1+t)^{1-\frac{n}{2m(\sigma-\sigma_1)}(p-1)+\frac{p\sigma_1}{\sigma-\sigma_1}}\lesssim 1,
\end{align*}
where we note that the condition 
\begin{align*}
    p \geq 1+ \frac{2m\sigma}{n-2m\sigma_1} &\text{ implies }  1-\frac{n}{2m(\sigma-\sigma_1)}(p-1) +\frac{p\sigma_1}{\sigma-\sigma_1} \leq 0.
\end{align*}
Next, the term $J_2(t)$ is controlled by
\begin{align*}
    J_2(t) &\lesssim (1+t)^{(p-1)(-\frac{n}{2(\sigma-\sigma_1)}(\frac{1}{m}-\frac{1}{q(p-1)})+\frac{\sigma_1}{\sigma-\sigma_1})} \int_{t/2}^{t} (1+t-\tau)^{-\frac{n}{2(\sigma-\sigma_1)}(\frac{1}{\gamma}-\frac{1}{2})+\frac{\sigma_1}{\sigma-\sigma_1}}d\tau,
\end{align*}
provided that it holds
\begin{align*}
&-\frac{n}{2(\sigma-\sigma_1)}\left(\frac{1}{m}-\frac{1}{q(p-1)}\right)+\frac{\sigma_1}{\sigma-\sigma_1}  \\
&\qquad = \left(-\frac{n}{2(\sigma-\sigma_1)} \left(\frac{1}{m}-\frac{1}{2}\right)+ \frac{\sigma_1}{\sigma-\sigma_1}\right) -\frac{n}{2(\sigma-\sigma_1)}\left(\frac{1}{2}-\frac{1}{q(p-1)}\right)< 0. 
\end{align*}
For this reason, carrying out some same steps as we did the term $J_1(t)$, we derive $J_2(t) \lesssim 1$. From this, the estimate (\ref{estimate1.3.2}) has been proved. Combining the estimates (\ref{estimate1.3.1}), 
 (\ref{estimate1.3.2}) and $\varepsilon$ is small enough, we conclude that $\mathcal{N}$ becomes a contraction mapping on $X(T, M\varepsilon)$ with respect to the metric of $Z(T)$. Then, using Banach’s fixed point theorem,  we derive a unique solution $u \in X(T,M\varepsilon)$. Since $T$ is arbitrary, it leads to our desired conclusion. Hence, the proof of Theorem \ref{Global-Existance} is completed.

\section{Blow-up results and Lifespan estimates}\label{Proof of blow-up results}
\subsection{Preliminaries}

In this section, our main aim is to prove Theorem \ref{Blow-up results}. To do this, we recall the definition of weak solution to (\ref{Main.Eq.1}).
\begin{definition} \label{defweaksolution_1}
Let $p>1$,  $T>0$ and $u_0 \equiv 0$. We say that $u \in \mathcal{C}([0,T), L^2)$ is a local weak solution to (\ref{Main.Eq.1}) if for any test functions $ \phi(t,x): = \eta(t)\varphi(x) $ with $\eta  \in \mathcal{C}_0^{\infty}[0,T)$, and for any $\varphi \in \mathcal{C}^{\infty}(\mathbb{R}^n)$ with all derivatives in $L^1 \cap L^{\infty}$, it holds
\begin{align}
&\int_0^T \int_{\R^n}|u(t,x)|^p \phi(t,x)dxdt \nonumber + \varepsilon \int_{\R^n} u_1(x) \phi(0,x)dx \nonumber \\
&\qquad= \int_0^T \int_{\R^n}u(t,x)\bigg(\partial_t^2 - \partial_t (-\Delta)^{\sigma_1}- \partial_t (-\Delta)^{\sigma_2} +(-\Delta)^{\sigma} \bigg)\phi(t,x)dxdt.\notag
\end{align}
If $T= \ity$, we say that $u$ is a global weak solution to \eqref{Main.Eq.1}. 
\end{definition}
\vspace{0.3cm}
Next, let us recall several known ingredients which will be used in the proof of Theorem \ref{Blow-up results}.

\begin{lemma}[see Corollary 3.1 in \cite{DAbbiccoFujiwara2021}]\label{Lemma2.1}
    Let $f(x)= \langle x \rangle^{-q}$ for $q > n$ and $\theta > 0$. Then, it holds for any $x \in \mathbb{R}^n$
    \begin{align*}
       |(-\Delta)^{\theta} f(x)| \lesssim 
        \begin{cases}
            \langle x \rangle^{-q-2\theta} &\text{ if } \theta \text{ is an integer},\\
            \langle x \rangle^{-n-2s} &\text{ if otherwise, where } s = \theta - [\theta]. 
        \end{cases}
    \end{align*}
\end{lemma}

\begin{lemma}[see Lemma 4 in \cite{DaoReissig2021}] \label{Lemma4.2}
Let $s\in (0,1)$. Assume that $\psi$ is a smooth function satisfying $\partial_x^2 \psi\in L^\ity$. For any $R>0$, let $\psi_R$ be a function defined by
$$ \psi_R(x)= \psi\big(R^{-1} x\big)\quad \text{ for all }x \in \R^n. $$
Then, $(-\Delta)^s (\psi_R)$ enjoys the following scaling property for all $x \in \R^n$:
$$(-\Delta)^s (\psi_R)(x)= R^{-2s}\big((-\Delta)^s \psi \big)\big(R^{-1} x\big). $$
\end{lemma}
 
\subsection{Proof of Theorem \ref{Blow-up results}}
At first, let us introduce the test functions $\eta= \eta(t)$ and $\varphi=\varphi(x)$ satisfying the following properties:
\begin{align*}
&1.\quad \eta \in \mathcal{C}_0^\ity([0,\ity)) \text{ and }
\eta(t)=
\begin{cases}
1 &\text{ if }0 \le t \le 1/2, \\
\text{decreasing} &\text{ if }1/2\le t\le 1, \\
0 &\text{ if }t \ge 1,
\end{cases} & \nonumber \\
&2.\quad \eta^{-\frac{p'}{p}}\big(|\eta'|^{p'}+|\eta''|^{p'}\big) \text{ is bounded, } & 
\end{align*}
where $p'$ stands for the conjugate of $p$. In addition, for $\epsilon$ is a sufficiently small positive constant, we define
\begin{align*}
    \varphi(x) := \langle x \rangle^{-n-\epsilon} = (|x|^2 +1) ^{-\frac{n+\epsilon}{2}}.
\end{align*}
 Let $R \geq 1$ and $K \geq 1$. We define the following test function:
$$ \Phi_R(t,x):= \eta_R(t) \varphi_R(x), $$
where $\eta_R(t):= \eta(R^{-\alpha}t)$, $\alpha > 0$ and $\varphi_R(x):= \varphi(R^{-1}K^{-1}x)$. Moreover, we define the functionals
\begin{align*}
 \mathcal{I}_R := \int_0^{\ity}\int_{\R^n}|u(t,x)|^p \Phi_R(t,x) dxdt= \int_{0}^{R^{\alpha}}\int_{\mathbb{R}^n}|u(t,x)|^p \Phi_R(t,x) dxdt 
 \end{align*}
 and
 \begin{align*}
     \tilde{\mathcal{I}}_R := \int_{R^\alpha/2}^{R^\alpha} \int_{\mathbb{R}^n} |u(t,x)|^p \Phi_R(t,x) dxdt.
 \end{align*}
Let us assume that $u= u(t,x)$ is a global weak solution from $\mathcal{C}([0, \infty), L^2)$ to (\ref{Main.Eq.1}) in the sense of Definition \ref{defweaksolution_1}. We replace the function $\phi(t,x)$ in Definition \ref{defweaksolution_1} by $\Phi_R(t,x)$ to derive
\begin{align}
    &\int_0^{\infty}\int_{\mathbb{R}^n} |u(t,x)|^p \Phi_R(t, x) dxdt +\varepsilon\int_{\mathbb{R}^n} u_1(x)\varphi_R(x)dx\notag\\
    &\qquad\quad= \int_0^{\infty}\int_{\mathbb{R}^n} u(t,x) \left(\partial_t^2 -\partial_t (-\Delta)^{\sigma_1}-\partial_t (-\Delta)^{\sigma_2}+(-\Delta)^{\sigma}\right)\Phi_R(t,x) dxdt.
    \label{equation1.5.1}
\end{align} 
 Using H\"older's inequality with $1/p + 1/p' = 1$, Lemmas \ref{Lemma2.1}-\ref{Lemma4.2} and the change of variables $\Tilde{t} = R^{-\alpha} t$, $\Tilde{x} = R^{-\alpha} K^{-1} x$,  we obtain the following chain of estimates for all $R \geq 1$ and $K \geq 1$:
\begin{align*}
    &\int_0^{\infty} \int_{\mathbb{R}^n} |u(t,x) \partial_t^2 \Phi_R(t,x)|dxdt \\
    &\hspace{1cm}= \int_{R^\alpha/2}^{R^\alpha} \int_{\mathbb{R}^n}\left|u(t,x) \partial_t^2 \Phi_R(t,x) \right|dxdt\\
    &\hspace{1cm}\lesssim \left(\int_{R^{\alpha}/2}^{R^\alpha} \int_{\mathbb{R}^n} |u(t,x)|^p \Phi_R(t,x) dxdt\right)^{\frac{1}{p}} \left(\int_{R^{\alpha}/2}^{R^\alpha} \int_{\mathbb{R}^n} |\partial_t^2 \Phi_R(t,x)|^{p'} \Phi_R^{-\frac{p'}{p}}(t,x) dxdt\right)^{\frac{1}{p'}}\\
    &\hspace{1cm}\lesssim \tilde{\mathcal{I}}_R^{\frac{1}{p}} R^{-2\alpha + \frac{n+\alpha}{p'}} K^{\frac{n}{p'}},\\
    &\int_0^{\infty} \int_{\mathbb{R}^n}\left|u(t,x) \partial_t (-\Delta)^{\sigma_1} \Phi_R(t,x)\right| dxdt\\
    &\hspace{1cm}= \int_{R^{\alpha}/2}^{R^\alpha} \int_{\mathbb{R}^n} \left|u(t,x) \partial_t (-\Delta)^{\sigma_1} \Phi_R(t,x)\right| dxdt\\
    &\hspace{1cm}\lesssim \left(\int_{R^{\alpha}/2}^{R^\alpha} \int_{\mathbb{R}^n} |u(t,x)|^p \Phi_R(t,x) dxdt\right)^{\frac{1}{p}} \\
    &\hspace{3cm} \times \left(\int_{R^{\alpha}/2}^{R^\alpha} \int_{\mathbb{R}^n} |\partial_t (-\Delta)^{\sigma_1} \Phi_R(t,x)|^{p'} \Phi_R^{-\frac{p'}{p}}(t,x) dxdt\right)^{\frac{1}{p'}}\\
    &\hspace{1cm}\lesssim \tilde{\mathcal{I}}_R^{\frac{1}{p}} R^{-\alpha-2\sigma_1 +\frac{n+\alpha}{p'}} K^{-2\sigma_1+\frac{n}{p'}},\\
    &\int_0^{\infty} \int_{\mathbb{R}^n}\left|u(t,x) \partial_t (-\Delta)^{\sigma_2} \Phi_R(t,x)\right| dxdt\\
    &\hspace{1cm}= \int_{R^{\alpha}/2}^{R^\alpha} \int_{\mathbb{R}^n} \left|u(t,x) \partial_t (-\Delta)^{\sigma_2} \Phi_R(t,x)\right| dxdt\\
    &\hspace{1cm}\lesssim \left(\int_{R^{\alpha}/2}^{R^\alpha} \int_{\mathbb{R}^n} |u(t,x)|^p \Phi_R(t,x) dxdt\right)^{\frac{1}{p}} \\
    &\hspace{3cm} \times \left(\int_{R^{\alpha}/2}^{R^\alpha} \int_{\mathbb{R}^n} |\partial_t (-\Delta)^{\sigma_2} \Phi_R(t,x)|^{p'} \Phi_R^{-\frac{p'}{p}}(t,x) dxdt\right)^{\frac{1}{p'}}\\
    &\hspace{1cm}\lesssim \tilde{\mathcal{I}}_R^{\frac{1}{p}} R^{-\alpha-2\sigma_2 +\frac{n+\alpha}{p'}} K^{-2\sigma_2 +\frac{n}{p'}}
\end{align*}
and
\begin{align*}
    &\int_0^{\infty} \int_{\mathbb{R}^n}\left|u(t,x) (-\Delta)^{\sigma} \Phi_R(t,x)\right| dxdt\\
    &\hspace{1cm} = \int_0^{R^\alpha}  \int_{\mathbb{R}^n}\left|u(t,x) (-\Delta)^{\sigma} \Phi_R(t,x)\right| dxdt\\
    &\hspace{1cm} \lesssim \left(\int_0^{R^\alpha} \int_{\mathbb{R}^n} |u(t,x)|^p \Phi_R(t,x) dxdt\right)^{\frac{1}{p}} \left(\int_0^{R^\alpha}\int_{\mathbb{R}^n} \left|(-\Delta)^{\sigma} \Phi_R(t,x)\right|^{p'} \Phi_R^{-\frac{p'}{p}}(t,x)dxdt\right)^{\frac{1}{p'}}\\
    &\hspace{1cm} \lesssim \mathcal{I}_R^{\frac{1}{p}} R^{-2\sigma+\frac{n+\alpha}{p'}} K^{-2\sigma +\frac{n}{p'}}.
\end{align*}
Let us fix $\alpha :=2(\sigma-\sigma_1)$. Combining this with the relations (\ref{equation1.5.1}) we conclude
\begin{align*}
    \mathcal{I}_R + \varepsilon\int_{\mathbb{R}^n} u_1(x) \varphi_{R}(x) dx \lesssim \mathcal{I}_R^{\frac{1}{p}} R^{-2\sigma+\frac{n+2(\sigma-\sigma_1)}{p'}} K^{\frac{n}{p'}}.
\end{align*}
  Moreover, the application of the inequality 
\begin{equation*}
    Ac^{\gamma}-c \leq A^{\frac{1}{1-\gamma}} \text{ for all } A > 0, c \geq 0 \text{ and } 0 < \gamma < 1
\end{equation*}
leads to
\begin{equation}\label{equation1.5.7}
    0 < \varepsilon\int_{\mathbb{R}^n}u_1(x) \varphi_R(x)dx \lesssim R^{-2\sigma p' + n+ 2(\sigma-\sigma_1)} K^n
\end{equation}
for all $R > R_0$. Let us now divide our consideration into the following two cases. \medskip

\begin{itemize}[leftmargin=*]
    \item \textbf{Case 1:} If $m = 1$, then we note that the assumption (\ref{condition1.1.2}) is equivalent to
    $$-2\sigma p' + n + 2(\sigma-\sigma_1) \leq 0. $$
    In the subcritical case, i.e.  $-2\sigma p' + n + 2(\sigma-\sigma_1) < 0$, taking $K= 1$ and letting $R \to \infty$ in (\ref{equation1.5.7}) we obtain
\begin{equation*}
    \int_{\mathbb{R}^n} u_1(x) dx = 0.
\end{equation*}
So, this is a contradiction to the assumption (\ref{condition1.1.1}). In the critical case, i.e. $-2\sigma p'+ n+ 2(\sigma-\sigma_1) = 0$, we see that $\mathcal{I}_R$  is uniformly bounded, that is, $u \in L^p([0, \infty) \times \mathbb{R}^n)$. Consequently, one gets
\begin{align*}
    \lim_{R\to\infty} \int_0^{\infty}\int_{\mathbb{R}^n} u(t,x) \left(\partial_t^2 -\partial_t (-\Delta)^{\sigma_1}-\partial_t (-\Delta)^{\sigma_2}\right)\Phi_R(t,x) dxdt = 0
\end{align*}
by noticing that $\dps\lim_{R\to\infty} \tilde{\mathcal{I}}_R = 0$. The employment of Young's inequality implies
\begin{align*}
    \mathcal{I}_R + \varepsilon\int_{\mathbb{R}^n} u_1(x) \varphi_R(x) dx \leq C' \mathcal{I}_R^{\frac{1}{p}} K^{-2\sigma+\frac{n}{p'}} \leq \frac{\mathcal{I}_R}{p} + C K^{n-2\sigma p'},
\end{align*}
that is,
\begin{align*}
    \lim_{R\to\infty} \mathcal{I}_R \leq -p' \varepsilon\int_{\mathbb{R}^n} u_1(x) dx + p'C K^{n-2\sigma p'}.
\end{align*}
Observing that $n-2\sigma p' < 0$ we pass $K \to \infty$ to derive a contradiction to the assumption (\ref{condition1.1.1}) again.
\item \textbf{Case 2:} If $m> 1$, then we use (\ref{equation1.5.7}) again by taking $K= 1$ and noticing the condition (\ref{condition1.1.1}) to give
\begin{align*}
    \int_{\mathbb{R}^n} u_1(x)\varphi_R(x)dx &\geq \int_{R/2 \leq |x| \leq R} |x|^{-\frac{n}{m}} (\log(e+|x|))^{-1}  \varphi_R(x)dx \\
    &\geq C \int_{R/2}^R |x|^{n-\frac{n}{m}-1} (\log(|x|))^{-1} d|x|\\
    &\geq C R^{n-\frac{n}{m}} (\log R)^{-1},
\end{align*}
which leads to
\begin{align}
   \varepsilon (\log R)^{-1} \lesssim R^{-2\sigma p'+\frac{n}{m}+2(\sigma-\sigma_1)}. \label{estimate2.3.100}
\end{align}
However, the fact is that the last estimate does not hold for any sufficiently large number $R$ when the following condition occurs:
$$-2\sigma p' +\frac{n}{m} +2(\sigma-\sigma_1) < 0. $$
\end{itemize} 
Summarizing, the proof of Theorem \ref{Blow-up results} is completed.

\subsection{Estimates for lifespan}
In this subsection, we will summarize how to get some estimates for lifespan of
solutions in the subcritical case
\begin{align}
    p < 1+ \frac{2m\sigma}{n-2m\sigma_1}. \label{subcritical}
\end{align}
Let us denote by $T_m(\varepsilon)$ the so-called lifespan of a local (in time) solution, i.e. the maximal existence time of local solutions. Then, lower bound estimates and upper bound estimates for $T_m(\varepsilon)$ are given by the next statements. \medskip

\begin{proposition}[\textbf{Lower bound of lifespan}]\label{lower-bound}
    Under the conditions of Theorem \ref{Local-existence}, we impose the additional assumption (\ref{subcritical}). Then, there exists a constant $\varepsilon_0 > 0$ such that for any $\varepsilon \in (0, \varepsilon_0]$, the lower bound for the lifespan $T_m(\varepsilon)$ can be estimated as follows:
    \begin{align}\label{Lower_Lifespan}
        T_m(\varepsilon) \gtrsim \varepsilon^{-\frac{p-1}{\gamma_m}},
    \end{align}
    where
    \begin{align}
        \gamma_m := 1-\frac{n}{2m(\sigma-\sigma_1)}(p-1) + \frac{p\sigma_1}{\sigma-\sigma_1}. \label{defi5.1.1}
    \end{align}
\end{proposition}
\begin{proof}
To prove Proposition \ref{lower-bound}, we will again use the definitions of the constant $M$ and the mapping $\mathcal{N}$ on the set $X(T, M \varepsilon)$ from the proof of Theorem \ref{Local-existence}, in particular, the estimates (\ref{Constant1}) and (\ref{map1}). From Proposition \ref{lemma3.1} and the condition (\ref{subcritical}), we can see that
\begin{align*}
   \left\|\int_0^t K_1(t-\tau, x) \ast_{x} |u(\tau,x)|^p d\tau\right\|_{X(T)} &\lesssim (1+T)^{1-\frac{n}{2m(\sigma-\sigma_1)}(p-1)+\frac{p\sigma_1}{\sigma-\sigma_1}} \||u|^p\|_{Y(T)} \\
   &\lesssim (1+T)^{1-\frac{n}{2m(\sigma-\sigma_1)}(p-1)+\frac{p\sigma_1}{\sigma-\sigma_1}} \|u\|_{X(T)}^p.
\end{align*}
For this reason, with two function $u, v$ belonging to $X(T, M\varepsilon)$ we arrive at
\begin{align*}
    \|\mathcal{N}[u]\|_{X(T)} \leq \frac{M\varepsilon}{2} + C_1 (1+T)^{1-\frac{n}{2m(\sigma-\sigma_1)}(p-1)+\frac{p\sigma_1}{\sigma-\sigma_1}} M^p \varepsilon^p
\end{align*}
    and
    \begin{align*}
        \|\mathcal{N}[u]-\mathcal{N}[v]\|_{Z(T)} \leq C_2 (1+T)^{1-\frac{n}{2m(\sigma-\sigma_1)}(p-1)+\frac{p\sigma_1}{\sigma-\sigma_1}} M^{p-1} \varepsilon^{p-1} \|u-v\|_{Z(T)},
    \end{align*}
    where constants $C_1$ and $C_2$ independent of $M, \varepsilon$ and $T$. Therefore, if we assume that
    \begin{align}
        \max\{C_1, C_2\} (1+T)^{1-\frac{n}{2m(\sigma-\sigma_1)}(p-1)+\frac{p\sigma_1}{\sigma-\sigma_1}} M^{p-1} \varepsilon^{p-1} < \frac{1}{4}, \label{condition5.1.1}
    \end{align}
    then $\mathcal{N}$ is a contraction mapping on $X(T, M\varepsilon)$ with the metric of $Z(T)$. As a consequence, we may construct a unique local solution $u \in X(T, M\varepsilon)$, moreover, the following estimate holds:
    \begin{align*}
        \|u\|_{X(T)} < \frac{3M\varepsilon}{4}.
    \end{align*}
    Let us choose
    \begin{align*}
        T_m^* := \sup\bigg\{T \in (0, T_m(\varepsilon)) \text{ such that } \mathcal{G}(T) := \|u\|_{X(T)} \leq M\varepsilon\bigg\}.
    \end{align*}
    If $T_m^*$ satisfies the inequality (\ref{condition5.1.1}), then from the previous estimate it follows that $\mathcal{G}(T_m^*) < 3M\varepsilon/4$. Due to the fact that $\mathcal{G}(T)$ is a continuous and increasing function for any $T \in (0, T_m(\varepsilon))$, there exists $T_m^0 \in (T_m^*, T_m(\varepsilon))$ such that $\mathcal{G}(T_m^0) \leq M\varepsilon$. This contradicts to the definition of $T_m^*$. For this reason, one realizes
    \begin{align*}
        \max\{C_1, C_2\} (1+T_m^*)^{1-\frac{n}{2m(\sigma-\sigma_1)}(p-1)+\frac{p\sigma_1}{\sigma-\sigma_1}} M^{p-1} \varepsilon^{p-1} \geq \frac{1}{4},
    \end{align*}
    that is,
    \begin{align*}
        T_m(\varepsilon) \geq T_m^* \gtrsim \varepsilon^{-\frac{p-1}{\gamma_m}}.
    \end{align*}
    Thus, Proposition \ref{lower-bound} has been proved.
\end{proof}

\begin{proposition}[\textbf{Upper bound of lifespan}]\label{upper-bound}
  Assume that the conditions of Theorem \ref{Blow-up results} are satisfied. Then, the upper bound for the lifespan $T_m(\varepsilon)$ can be estimated by
  \begin{align}
      T_m(\varepsilon) &\lesssim
      \begin{cases}
          \varepsilon^{-\frac{p-1}{\gamma_1}} &\text{ if } m = 1,\\
          \varepsilon^{-\frac{p-1}{\gamma_m}-\delta} &\text{ if } m > 1,
      \end{cases} \label{estimate5.2.1}
   \end{align}
   where $\gamma_m$ is defined as in (\ref{defi5.1.1}) and $\delta$ is an arbitrarily small positive real number.
\end{proposition}
\begin{proof}
   In the case $m=1$, using Lebesgue's dominated convergence theorem we gain 
    \begin{align*}
        \lim_{R\to\infty} \int_{\mathbb{R}^n} u_1(x) \varphi_R(x) dx = \int_{\mathbb{R}^n} u_1(x) dx.
    \end{align*}
    For this reason, with $R_1$
  being a sufficiently large positive real number, we have the following relation:
  \begin{align*}
      \int_{\mathbb{R}^n} u_1(x) \varphi_R(x) dx > \frac{1}{2} \int_{\mathbb{R}^n} u_1(x) dx =: C_{u_1} &\text{ for all } R > R_1.
  \end{align*}
  From the relation (\ref{equation1.5.7}), we take $K = 1$ and pass $R \to T_{1}(\varepsilon)^{\frac{1}{2(\sigma-\sigma_1)}}$ to obtain
  \begin{align*}
      C_{u_1} \varepsilon \leq C T_1(\varepsilon)^{\frac{-2\sigma p'+n+2(\sigma-\sigma_1)}{2(\sigma-\sigma_1)}}.
  \end{align*}
   Therefore, we may claim the estimate (\ref{estimate5.2.1}) when $m=1$. In the case $m >1$, we will again use the estimate (\ref{estimate2.3.100}) by letting $R \to T_{m}(\varepsilon)^{\frac{1}{2(\sigma-\sigma_1)}}$ to conclude the desired estimate.
\end{proof}

\begin{remark}
\fontshape{n}
\selectfont
    Linking the achieved estimates \eqref{Lower_Lifespan} and \eqref{estimate5.2.1} in Propositions \ref{lower-bound} and \ref{upper-bound} one recognizes that the lifespan estimates for solutions to the Cauchy problem \eqref{Main.Eq.1} in the subcritical case, i.e. the condition \eqref{subcritical} occurs, are determined by the following relations:
    $$T_1(\varepsilon) \sim \varepsilon^{-\frac{p-1}{\gamma_1}}$$
    for $m=1$, and moreover,
    $$\varepsilon^{-\frac{p-1}{\gamma_m}} \lesssim T_m(\varepsilon) \lesssim \varepsilon^{-\frac{p-1}{\gamma_m}-\delta} $$
    for $m \in (1, 2]$, where $\delta$ is an arbitrarily small positive real number. The sharp lifespan estimates for solutions in the subcritical case for $m \in (1, 2]$ is really a challenging problem, which will be partially investigated in our forthcoming work.
\end{remark}
\section*{Acknowledgments}
This research is funded by Vietnam National Foundation for Science and Technology Development (NAFOSTED) under grant number 101.02-2023.12.



\begin{thebibliography}{00}
 \bibliographystyle{plain}
 

 \bibitem{ChenDAbbiccoGirardi2022}  Chen, W.,  D'Abbicco, M., Girardi, G.: Global small data solutions for semilinear waves with two dissipative terms. Ann. Mat. Pura. Appl. \textbf{201}, 529--560 (2022)
 
\bibitem{DAbbiccoReissig2014} D'Abbicco, M., Reissig, M.: Semilinear structural damped waves. Math. Methods Appl. Sci. \textbf{37}, 1570--1592 (2014)

\bibitem{DAbbiccoEbert2021}   D'Abbicco, M., Ebert, M.R.: $L^p - L^q$ estimates for a parameter-dependent multiplier
with oscillatory and diffusive components. J. Math. Anal. Appl. \textbf{504}, 125393 (2021)

\bibitem{DAbbiccoEbert2022} D'Abbicco, M., Ebert, M.R.: The critical exponent for semilinear $\sigma$-evolution equations with a
strong non-effective damping. Nonlinear Anal. \textbf{215}, 112637 (2022)

\bibitem{DAbbiccoFujiwara2021} D'Abbicco, M., Fujiwara, K.: A test function method for evolution equations with fractional powers of the Laplace operator. Nonlinear Anal. \textbf{202}, 112114 (2021)

\bibitem{DaoDuongNguyen2024} Dao, T.A., Duong, D.V., Nguyen, D.A.: On asymptotic properties of solutions to $\sigma$-evolution equations
with general double damping. J. Math. Anal. Appl. \textbf{536}, 128246 (2014)

\bibitem{DaoMichihisa2020} Dao, T.A., Michihisa, H.: Study of semi-linear $\sigma$-evolution equations with frictional and visco-elastic damping. Commun. Pure Appl. Anal. \textbf{19}, 1581-1608 (2020)

\bibitem{DaoReissig2019_1} Dao, T.A., Reissig, M.: An application of $L^1$ estimates for oscillating integrals to parabolic like semi-linear structurally damped $\sigma$-evolution models. J. Math. Anal. Appl. \textbf{476}, 426--463 (2019)

\bibitem{DaoReissig2019_2} Dao, T.A., Reissig, M.: $L^1$ estimates for oscillating integrals and their applications to semi-linear models with $\sigma$-evolution like structural damping. Discrete Contin. Dyn. Syst. A. \textbf{39}, 5431--5463 (2019)

\bibitem{DaoReissig2021} Dao, T.A., Reissig, M.: Blow-up results for semi-linear structurally damped $\sigma$-evolution equations, In: M. Cicognani, D. Del Santo, A. Parmeggiani, M. Reissig (eds.), Anomalies in Partial Differential Equations, Springer INdAM Series, \textbf{43}, 213--245 (2021)

\bibitem{DuongKainaneReissig2015} Duong, P.T., Kainane, M.M., Reissig, M.: Global existence for semi-linear structurally damped $\sigma$-evolution models. J. Math. Anal. Appl. \textbf{431}, 569--596 (2015)

\bibitem{Hayashi2004} Hayashi, N., Kaikina, E.I., Naumkin, P.I.: Damped wave equation with super critical nonlinearities. Differ. Integr. Equ. \textbf{17}, 637- 652 (2004)

\bibitem{Hayashi2006} Hayashi, N., Kaikina, E.I., Naumkin, P.I.: Damped wave equation with a critical nonlinearity. Trans. Amer. Math. Soc. \textbf{358}, 1165-1185 (2006)

\bibitem{Hayashi2017} Hayashi, N., Naumkin, P.I.: On the critical nonlinear damped wave equation with large initial data. J. Math. Anal. Appl. \textbf{334}, 1400-1425 (2007)

\bibitem{Ozawa}  Hajaiej, H., Molinet, L., Ozawa, T., Wang, B.:  Necessary and sufficient conditions for the fractional Gagliardo-Nirenberg inequalities and applications to Navier-Stokes and generalized boson equations (Harmonic analysis and nonlinear partial differential equations). RIMS Kokyuroku Bessatsu, B26, Res. Inst. Math. Sci. (RIMS). Kyoto,  159--175 (2011)

\bibitem{Ikeda2019} Ikeda, M., Inui, T., Okamoto, M., Wakasugi, Y.: $L^p - L^q$ estimates for the damped wave equation and the critical exponent for the nonlinear problem with slowly decaying data. Commun. Pure Appl. Anal. \textbf{18} 1967-2008 (2019)

\bibitem{IkehataOhta2002} Ikehata, R., Ohta, M.: Critical exponents for semilinear dissipative wave equations in $\mathbb{R}^N$. J. Math. Anal. Appl. \textbf{269},  87–97 (2022)

\bibitem{IkehataTakeda2017} Ikehata, R., Takeda, H.: Critical exponent for nonlinear wave equations with frictional and viscoelastic damping terms. Nonlinear Anal. \textbf{148}, 228--253 (2017)

\bibitem{IkehataTakeda2019} Ikehata, R., Takeda, H.: Asymptotic profiles of solutions for structural damped wave equations. J. Dynam. Differential Equations. \textbf{31}, 537-571 (2019)

\bibitem{Matsumura1976} Matsumura, A.: On the asymptotic behavior of solutions of semi-linear wave equations. Publ. Res. Inst. Math. Sci. \textbf{12}, 169–189 (1976) 

\bibitem{NakaoOno1993} Nakao, M., Ono, K.: Existence of global solutions to the Cauchy problem for the semilinear dissipative wave equation. Math. Z. \textbf{214}, 325-342 (1993)

\bibitem{TodorovaYordanov2001} Todorova, G., Yordanov, B.: Critical Exponent for a Nonlinear Wave Equation with Damping. J. Differential Equations. \textbf{174}, 464-489 (2001)

\end{thebibliography}
\end{document}